\RequirePackage[british]{babel}
\documentclass[a4paper,leqno,twoside]{amsart}
\usepackage{amssymb}
\usepackage{amscd}
\usepackage{url}
\usepackage{tabularx}
\usepackage[all]{xy}

\theoremstyle{plain}
\newtheorem{thm}{Theorem}

\newtheorem{prop}[thm]{Proposition}
\newtheorem{lem}[thm]{Lemma}

\theoremstyle{definition}
\newtheorem{df}{Definition}

\theoremstyle{remark}
\newtheorem{rmk}{Remark}
\newtheorem*{acks}{Acknowledgments}

\newcommand{\ie}{\textit{i.e. }}

\newcommand{\cL}{\mathcal L}
\newcommand{\p}{\mathbb{P}}

\newcommand{\R}{\mathcal{R}}

\newcommand{\NN}{\mathbb{N}}
\newcommand{\FF}{\mathbb{F}}

\newcommand{\G}{\mathrm{Gr}}

\newcommand{\OO}{\mathcal{O}}

\newcommand{\II}{\mathcal{I}}

\newcommand{\Sym}{\mathcal{S}ym}
\newcommand{\abs}[1]{\lvert#1\rvert}

\newcommand{\gen}[1]{\langle#1\rangle}
\newcommand{\Z}{{\mathbb Z}}
\newcommand{\CC}{{\mathbb C}}
\newcommand{\de}{\partial}

\DeclareMathOperator{\pic}{Pic}

\DeclareMathOperator{\coker}{coker}

\DeclareMathOperator{\Ext}{Ext}

\def\cocoa{{\hbox{\rm C\kern-.13em o\kern-.07em C\kern-.13em o\kern-.15em A}}}
\DeclareMathOperator{\PGL}{\mathbb{P}GL}

\DeclareMathOperator{\proj}{Proj}

\begin{document}

\title{Fermat hypersurfaces and Subcanonical curves}
\author{Pietro De Poi \and Francesco Zucconi}

\thanks{}

\address{
Dipartimento di Matematica e Informatica\\
Universit\`a degli Stud\^\i\ di Udine\\
Via delle Scienze, 206\\
Loc. Rizzi\\
33100 Udine\\
Italy
}

\email{pietro.depoi@dimi.uniud.it} \email{francesco.zucconi@dimi.uniud.it}
\date{\today}

\begin{abstract} 
We show that any Fermat hypersurface of degree $s+2$ is apolar to a
$s$-subcanonical $(s+2)$-gonal projectively normal curve, and vice versa.

Moreover, we extend the classical Enriques-Petri Theorem to $s$-subcanonical projectively normal curves, 
proving that such a curve is $(s+2)$-gonal if and only if it is contained in a rational normal surface. 
\end{abstract}
\maketitle

\bibliographystyle{amsalpha}

\section{Introduction} 

A nice result by Macaulay \cite{Mac} is that an Artinian graded Gorenstein ring of
socle dimension $1$ and degree $s+2$ can be
realised as $A=\frac{\CC[\de_0,\dotsc,
\de_N]}{F^\perp}$, where $F\in \CC[x_0,\dotsc, x_N]$ 
is a homogeneous polynomial of degree $s+2$ and $F^\perp:=
\{ D\in \CC[\de_0,\dotsc, 
\de_N]\vert\ D(F)=0 \}$, where  $\CC[\de_0,\dotsc,
\de_{N}]$ is the polynomial ring generated by the natural
derivations over $\CC[x_0,\dotsc,
x_{N}]$.

In this paper, we 
characterise those Artinian graded Gorenstein ring of
socle degree $s+2$ such that $F$ is a Fermat hypersurface.

To state our 
main results, we consider an embedding $j\colon C\to
X\subset \p^{N+2}$ of a smooth projective curve $C$ 
such that the homogeneous coordinate ring $S_X$ of its image $X$ is integrally
closed, that is we assume that $X$ is \emph{projectively normal}. 
We say that $X$ is \emph{subcanonical} if the dualising sheaf
$\omega_C$ is isomorphic to $j^{\star}\OO_{\p^{N+2}}(s)$, where $s\in \Z$.
The ring $S_X$ for projectively normal subcanonical curves works perfectly to
produce Artinian graded Gorenstein ring of 
socle dimension $1$. In fact it is not hard to show that, letting
$ \p^{N+2}:= \proj(\CC[\de_0,\dotsc,\de_{N+2}])$, and fixing $\eta_1,\eta_2$ two general
linear forms in $\CC[\de_0,\dotsc,\de_{N+2}]$, that can be assumed to be $\eta_1=\de_{N+1}$ and
$\eta_2=\de_{N+2}$, then the ring $A:=
\frac{S_X}{\gen{\eta_1,\eta_2}}$ is an Artinian graded Gorenstein ring of
socle dimension $1$. Hence, by the above mentioned Macaulay's result, $A$ can be
realised as $A=A^F=\frac{\CC[\de_0,\dotsc,\de_N]}{F^\perp}$
where $F\in \CC[x_0,\dotsc, x_N]$ 
is a homogeneous
polynomial. 

Hence the reader can see immediately that, given a projectively normal, subcanonical
curve $X\subset\p^{N+2}$, it remains defined a rational map:
\begin{equation}\label{eq:ac} 
\alpha_X\colon\G(N,N+2)\dashrightarrow H_{N,s}
\end{equation}
where $\G(N,N+2)$ is the Grassmannian of $N$-planes in $\p^{N+2}$ and
$H_{N,s}$ is the space of homogeneous polynomials of degree $s+2$
in $\check\p^N$ modulo the action of $\PGL(N,\CC)$, by the map
$\gen{\eta_1,\eta_2}\mapsto [F_{\eta_1,\eta_2}]:=[F]$.

In \cite{DZ} we have shown that a curve $C$ of genus $g\geq 3$ is either trigonal or isomorphic 
to a smooth plane quintic if and only if for every
couple of general sections
$\eta_1,\eta_2\in H^{0}(C,\omega_{C})$, $F_{\eta_{1},\eta_{2}}$ 
is a Fermat cubic. This characterisation of Fermat hypercubics was
obtained using the canonical embedding: \ie we considered the case $s=1$. 
Motivated by this result of \cite{DZ}, we say that a form $F$ is \emph{apolar} to a subcanonical 
curve $X$ if there exist two sections $\eta_1,\eta_2$ of the
$s$-subcanonical system such that $F=F_{\eta_1,\eta_2}$. 

In this paper, see
Theorem~\ref{miofratellos}, we prove:

\begin{thm}\label{miocugino}
Let $(C,\cL)$ be a polarised curve, such that $C\subset\abs{\cL}^\vee=:\check\p^N$ is a $s$-subcanonical 
projectively normal curve. 
Then $C$ is $(s+2)$-gonal if and only if for every couple of general sections
$\eta_1, \eta_2\in H^{0}(C,\cL)$, $F_{\eta_{1},\eta_{2}}$
is a Fermat $(s+2)$-tic. 
\end{thm}

and moreover, 
again see Theorem~\ref{miofratellos}:

\begin{thm}\label{miacugina}
Let $(C,\cL)$ be a polarised curve, such that $C\subset\abs{\cL}^\vee=:\check\p^N$ is a $s$-subcanonical 
projectively normal curve. 
Then $C$ is $(s+2)$-gonal if and only if $C$ is contained in a rational normal surface
$S_{a_1,a_2}$ with $a_{2} \leq a_{1}\leq (s+1)a_{2}+2$ and $N:=a_1+a_2+1$.
\end{thm}

The above theorems fit in a long path started by classical algebraic
geometers.  In fact, Theorem  \ref{miacugina} could be seen as a
natural generalisation of the trigonal case of the
Enriques-Petri Theorem (see \cite{E} for the original work of 
Enriques, completed by \cite{Ba}, and \cite{P}; modern versions are in \cite{So} and \cite{SD}) to the case of
$s$-subcanonical curves. 

Instead Theorem \ref{miocugino}, see the proof of Theorem 
\ref{miofratellos},
is an evidence
of the actual possibility to rewrite in a more algebraic terms at
least a part of the theory of curves using notions as Artinian Gorenstein
rings, Fermat's
hypersurfaces, etc. as substitutes of canonical ring, varieties of minimal
degree, etc. We note moreover that we could prove the classical Enriques-Petri Theorem 
with the same techniques used to prove~\ref{miofratellos}.

We also study in Proposition \ref{miazia} and in Proposition \ref{miazias} 
large classes of projectively normal subcanonical curves to estimate
the width of the class of curves to apply our theory.

Another aspect of our theory is to relate apolarity to plane curves.
Proposition \ref{cuginettoo} and Theorem \ref{cuginetto} are results
in this direction.

Finally, we want to stress that in this paper we use standard 
techniques of algebraic geometry, but the link that apolarity
establishes between the theory of Fermat hypersurfaces and the theory 
of $n$-gonal curves contained into varieties of minimal degree opens a range of 
ideas which are new and still not fully explored.

\begin{acks}
We would like to thank E. Ballico and G. Casnati for various remarks and suggestions. 
\end{acks}

\section{Preliminaries}



In this paper we will work with varieties and schemes over the complex field $\CC$. For us, 
a \emph{variety} will always be irreducible. 

\subsection{Arithmetic Gorenstein schemes }\label{ssec:ags}

Let us fix a closed subscheme $X$ of $\p^N$ of dimension $n\ge1$ and a
system $x_0,\dotsc,x_N$ of projective coordinates. Moreover, 
let us---as usual---denote by $\II_X$ the \emph{sheaf of ideals} of $X$ and by 
$M^r(X) := \oplus_{t\in\Z} H^r (\II_X(t))$, $1\le r\le n$, 
the \emph{$r$-th Hartshorne-Rao (or deficiency) module} of $X$. 
We recall the following

\begin{df}
Let $X$ be as above. Then, $X$ is said to be
\emph{arithmetically Cohen-Macaulay} (\emph{aCM} for short) 
if $N-n$ is equal to the length of a minimal free
resolution of its homogeneous coordinate ring
\begin{equation*}
S_X:=\frac{S}{M^0(X)}
\end{equation*}
as an $S$-module, where we have set $S:=\CC[x_0,\dotsc,x_N]$ and 
$M^0(X) := \oplus_{t\in\Z} H^0 (\II_X(t))$.
\end{df}

\begin{rmk}\label{rmk:1}
It can be proved that $X$ is aCM if and only if all the Hartshorne-Rao modules $M^r (X)$ vanish. 
(see for example \cite[1.2.2 and 1.2.3]{Mig}).
\end{rmk}

Then, we recall the fundamental 

\begin{df}
Let $X$ be as above. $X$ is said to be \emph{arithmetically Gorenstein} (\emph{aG} for short) if it is aCM 
and the last free module of a minimal free resolution of $S_X$ has rank $1$. 
\end{df}

Also we recall

\begin{df}\label{df:can}
 Any subscheme $X\subset\p^N$ is said to be \emph{subcanonical} if there exists 
an integer $s\in \Z$ such that $\omega_X\cong\OO_X(s)$, where $\omega_X$ is the dualising sheaf of $X$. 
\end{df}

We stress that the last two definitions are indeed equivalent for aCM schemes; in fact, denoting by 
\begin{equation*} 
K_X:=\Ext_S^{n-r}(S_X,S)(-n -1) = \Ext_S^{n-r}(S_X, K_{\p^N}) 
\end{equation*}
the \emph{canonical module} of $X$, it holds: 

\begin{prop}\label{senzanome}
If $X$ is an aCM closed subscheme of $\p^N$, then the following are equivalent:
\begin{enumerate}
\item $X$ is aG;
\item \label{senzanome2} $S_X\cong K_X(-s)$ for some integer $s$;
\item \label{senzanome3} the minimal free resolution of $S_X$ is self-dual, up to a twist. 
\end{enumerate}
\end{prop}
\begin{proof} 
See \cite[Proposition 4.1.1]{Mig}.
\end{proof}

\begin{rmk}
To point out the integer $s$ of Definition~\ref{df:can}, we will say also 
that $X$ is \emph{$s$-subcanonical} or \emph{$s$-arithmetically Gorenstein} (and so, \emph{$s$-aG} for short).
\end{rmk}

\begin{df}
A projective closed subscheme $X\subset\p^N$ is said to be
\emph{projectively normal} (\emph{PN} for short) if $S_X$ 
is integrally closed. 
\end{df}

\begin{df}
Let $X\subset\p^N$ be a projective closed subscheme. We say that $X$ is \emph{$j$-normal}, 
with $j\in \Z$, $j\ge 0$, if the natural restriction 
map
\begin{equation*}
H^0(\OO_{\p^N}(j))\to H^0(\OO_X(j))
\end{equation*}
is surjective. 
\end{df}

\begin{rmk}
It is obvious that $X$ is $j$-normal if and only if $h^1(\II_X(j))=0$. Moreover, $X$ is PN if and only if 
it is $j$-normal for all $j\in \Z$, $j\ge 0$. 
In particular, an aCM scheme is PN by Remark~\ref{rmk:1}. 
\end{rmk}

Vice versa, it is not difficult to show that 
\begin{prop}\label{prop:pn}
A PN closed $n$-dimensional subvariety $X\subset\p^N$ is aCM if and only if 
\begin{equation*}
h^i(\OO_X(j))=0,\quad\textup{for } 0< i < n \textup{ and } \forall j\in\Z.
\end{equation*}
In particular, a PN curve is always aCM, and if the curve is subcanonical, then it is also aG. 
\end{prop}

\begin{proof}
See for example \cite[1.5]{bk}. 
\end{proof}





\subsection{Apolarity}
Let $S:=\CC[x_0,\dotsc, x_N]$ be the polynomial ring in $(N+1)$-variables. 
The algebra of the partial derivatives on $S$, 
\begin{equation*}
T:=\CC[\de_0,\dotsc,\de_N],\qquad \de_i:=\frac{\de}{\de_{x_i}},
\end{equation*}
can act on the monomials in the following way:
\begin{equation*}
\de^a \cdot x^b=
\begin{cases}a!\binom{b}{a}x^{b-a} & \text{if $b\ge a$}\\
0 & \text{otherwise}
\end{cases}
\end{equation*}
where $a,b$ are multiindices $\binom{b}{a}=\prod_i\binom{a_i}{b_i}$, etc. 

We can think of $S$ as the algebra of partial derivatives on $T$ 
by defining
\begin{equation*}
x^a \cdot \de^b=
\begin{cases}a!\binom{b}{a}\de^{b-a} & \text{if $b\ge a$}\\
0 & \text{otherwise.}
\end{cases}
\end{equation*}

These actions define a perfect paring between the 
homogeneous 
forms in degree $d$ in $S$ 
and $T$ 
\begin{equation*}
S_d\times T_d\xrightarrow{\cdot} \CC.
\end{equation*}
Indeed, this is nothing but the extension of the duality between vector spaces: 
if $V:=S_1$, then $T_1=V^*$. 

Moreover, the perfect 
paring shows the natural duality between $\p^N:=\proj(S)$ and 
$\check\p^N=\proj(T)$. 
More precisely, if $(c_0,\dotsc, c_N)\in\check\p^N$, this gives 
$f_c:=\sum_i c_i x_i\in S_1$, and if $D\in T_a$, 
\begin{equation*}
D \cdot f_c^b=
\begin{cases}a!\binom{b}{a}D(c)f_c^{b-a} & \text{if $b\ge a$}\\
0 & \text{otherwise}.
\end{cases}
\end{equation*}
in particular, if $b\ge a$
\begin{equation*}\label{eq:lin}
0=D \cdot f_c^b \iff D(c)=0.
\end{equation*}

\begin{df}
We say that two forms, $f\in S$ and $g\in T$ are \emph{apolar} if 
\begin{equation*}
g \cdot f =f\cdot g =0.
\end{equation*}
\end{df}

Let $f\in S_d$ and $F:=V(f)\subset\p^N$ the corresponding hypersurface; 
let us now define 
\begin{equation*}
F^\perp:=\{D\in T\mid D\cdot f=0 \}
\end{equation*}
and
\begin{equation*}
A^F:=\frac{T}{F^\perp}.
\end{equation*}

\begin{lem}
The ring $A^F$ is Artinian Gorenstein of socle 
of 
dimension one and degree $d$. 
\end{lem}
\begin{proof}
See \cite[\S 2.3 page 67]{ik}.
\end{proof}

\begin{df}
$A^F$ is called the \emph{apolar} Artinian Gorenstein ring of $F$.
\end{df}

It holds the \emph{Macaulay Lemma}, that is 
\begin{lem}
The map 
\begin{equation*}
F\mapsto A^F
\end{equation*}
is a bijection between the hypersurfaces $F\subset\p^N$ of degree 
$d$ and graded Artinian Gorenstein quotient rings 
\begin{equation*}
A:=\frac{T}{I}
\end{equation*}
of $T$ with socledegree $d$. 
\end{lem}
\begin{proof}
See \cite[Lemma 2.12 page 67]{ik}.
\end{proof}

\subsubsection{Varieties of sum of powers}

Consider a hypersurface $F=V(f)\subset\p^N$ of degree $d$. 
\begin{df}
A subscheme $\Gamma\subset\check\p^N$ is said to be \emph{apolar to $F$} if
\begin{equation*}
\II(\Gamma)\subset F^\perp.
\end{equation*}
\end{df}

It holds the \emph{Apolarity Lemma}:

\begin{lem}\label{lem:apo}
Let us consider the linear forms $\ell_1,\dotsc,\ell_s\in S_1$ and 
let us denote by 
$L_1,\dotsc,L_s\in\check\p^N$ the corresponding points in the dual space. 
Then 
\begin{equation*}
\Gamma\ \textup{is apolar to}\ 
F 
= 
V(f),
\iff \exists \lambda_1,\dotsc,\lambda_s \in \CC^*\ \textup{such that}\ f=\lambda_1\ell_1^d+\dotsc+\lambda_s\ell_s^d
\end{equation*}
where $\Gamma:=\{L_1,\dotsc,L_s\}\subset \check\p^N$. If $s$ is minimal, then it is called the 
\emph{Waring number} of $F$. 
\end{lem}

\begin{proof}
See \cite[Lemma 1.15 page 12]{ik}.
\end{proof}



\subsection{Hypercubics and canonical sections}
In \cite{DZ}, 
we studied the special case of the canonical curve 
$C\subset\check\p^{g-1}$ of the map introduced in \eqref{eq:ac}.
In fact it is a well-known result that $C$ is PN (see \cite[page 117]{ACGH}), 
and then, by Proposition~\ref{prop:pn}, the canonical curve $C$ is aG. 
Therefore, if we take two general linear forms 
$\eta_1,\eta_2\in ({\R_C})_1=H^0(\omega_C)$, then 
$T:= \frac{\R_C}{\gen{\eta_1,\eta_2}}$ 
is Artinian Gorenstein, and its values of the Hilbert function are 
$1, g-2, g-2, 1$. In particular, the socledegree of $T$ is 
$3$, and by the Macaulay Lemma, this defines a hypercubic in 
$\proj(T^*)$. In this way we obtain the rational map 
$\alpha_C\colon\G(g-2,g)\dashrightarrow H_{g-3,3}$. 
In this paper we generalise this idea to the 
$s$-subcanonical curves.



\section{Subcanonical curves}

Let $(C,\cL)$ be a 
\emph{polarised curve}, \ie $C$ is a smooth curve and $\cL$ is a 
line bundle on $C$. We will suppose also that the complete linear system $\abs{\cL}$ \emph{embeds} $C$ in 
$\p^N=\abs{\cL}^*$, \ie that $\cL$ is very ample. 
Therefore, $(C,\cL)$ is subcanonical if there exists an $s\in \Z$ 
such that $\cL^{\otimes s}\cong\omega_C$. 

\begin{rmk}
Since we request $\cL$ to be very ample, if $C$ is also subcanonical, 
\emph{$C$ 
cannot be hyperelliptic}. 
\end{rmk}

Clearly, the only possible $s\le 0$ are $s=-2,-1$ 
and $C$ is rational, or $s=0$, and $C$ is elliptic. If $s=1$, 
then $C$ is a canonical curve. 
Therefore, from now on, for simplicity, we will say that $(C,\cL)$ is \emph{subcanonical} 
if there exists an $s\in \NN$, $s>1$, such that $\cL^{\otimes s}\cong\omega_C$. 
If $s=2$, we will talk of \emph{half-canonical} curves.
In other words, we restrict Definition~\ref{df:can} to $s$-subcanonical curves with $s\ge 2$. 

As we have observed in Subsection~\ref{ssec:ags}, a subcanonical PN curve is always aG, 
therefore the construction made for the canonical curves can be extended to these curves.

Let us show this in the case of the half-canonical curves. We have supposed that $\cL$ is very
ample and that embeds $C$ in $\p^N$, and we suppose also that $C$ is \emph{PN, with this embedding}. 
Clearly, by the theorem of Clifford (see for example \cite[IV.5.4]{H}) 
\begin{equation*}
N\le\frac{g-1}{s},
\end{equation*}
with equality only if $s=1$, \ie $\cL=\omega_C$.

To ease reading we first consider the case $s=2$.

\begin{prop}\label{prop:4}
Let $(C,\cL)$ be a half canonical curve, embedded by the very ample, complete, linear system $\abs{\cL}$ in
$\p^N$ and with image $X$. 
Let $\eta_1$ and $\eta_2$ two general linear forms on $\p^N$. Let us suppose that
$X$ is PN. 
Then, $X$ is aG and the values of the Hilbert function of the Artinian Gorenstein graded $\CC$-algebra 
$A:=\frac{S_X}{\gen{\eta_1,\eta_2}}$ 
are $1, N-1,g-1-2N, N-1, 1$. In particular, $A$ is Artinian Gorenstein of socle 
of dimension one and degree $4$.
\end{prop}

\begin{proof}
We can write the homogeneous coordinate ring of $X$ as
\begin{equation*}
S_X=\frac{\Sym(H^0(C,\cL))}{M^0(C)}, 
\end{equation*}
and the Artinian Gorenstein algebra $A$ as 
\begin{equation*}
A=\frac{S_X}{\gen{\eta_1,\eta_2}}, 
\end{equation*}
where $\eta_1$ and $\eta_2$ are two (general) linear forms. 

Let us consider the following exact sequence of 
sheaves
\begin{equation*}
0\to \cL^{\otimes i}\to \gen{\eta_1,\eta_2}\otimes \cL^{\otimes i+1}\to\cL^{\otimes 1+2}\to 0, 
\end{equation*}
with $i\in \Z$, $i\ge 0$. By the Base-Point-Free Pencil trick 
(see \cite[page 126]{ACGH}) the sequence gives, in cohomology, 
\begin{multline}\label{eq:c0}
i=0:\ 0\to H^0(\OO_C)\to \gen{\eta_1,\eta_2}\otimes H^0(\cL)\to H^0(\omega_C)\to\\ 
\to H^1(\OO_C)\to \gen{\eta_1,\eta_2}\otimes H^1(\cL)\to H^1(\omega_C)\to 0
\end{multline}
\begin{multline}\label{eq:c1}
i=1:\ 0\to H^0(\cL)\to \gen{\eta_1,\eta_2}\otimes H^0(\omega_C)\to H^0(\cL^{\otimes 3})\to\\ 
\to H^1(\cL)\to \gen{\eta_1,\eta_2}\otimes H^1(\omega_C)\to 0  
\end{multline}
\begin{align}
\label{eq:c2}
i&=2:\ 0\to H^0(\cL^{\otimes 2})\to \gen{\eta_1,\eta_2}\otimes H^0(\cL^{\otimes 3})\to H^0(\cL^{\otimes 4})\to 
H^1(\omega_C)\to 0\\ \label{eq:c3}
i&>2:\ 0\to H^0(\cL^{\otimes i})\to \gen{\eta_1,\eta_2}\otimes H^0(\cL^{\otimes i+1})\to H^0(\cL^{\otimes i+2})\to 0, 
\end{align}
with $h^0(\omega_C)=g$, $h^1(\cL)=h^0(\cL)=N+1$.

Moreover, the standard exact sequence of sheaves 
\begin{equation*}
0\to \II_X(i)\to \OO_{\p^N}(i)\to \OO_X(i)\to 0, 
\end{equation*}
with $i\in \NN$, gives rise to the following sequence in cohomology 
\begin{equation}\label{eq:ci2}
0\to I_X(i)\to \Sym^i(H^0(\cL))\to H^0(\cL^{\otimes i})\to 0, 
\end{equation}
Since $X$ is PN by hypothesis. 
Now, let us write $A$ as a graded algebra
$A=\bigoplus_i A_i$; 
then 
\begin{equation*}
A_i=\frac{\Sym^i(H^0(C,\cL))}{\gen{I_X(i),(\eta_1,\eta_2)\otimes \Sym^{i-1}(H^0(\cL)}}; 
\end{equation*}
we have clearly that 
\begin{equation*}
A_0\cong \CC,\quad
A_1=\cong \frac{\CC[x_0,\dotsc,x_N]_1}{\gen{\eta_1,\eta_2}}
\end{equation*}
\ie $A_1\cong \CC^{N-1}$. 
For the rest, putting together sequences \eqref{eq:c0},\eqref{eq:c1}, \eqref{eq:c2} and \eqref{eq:c3}, 
and \eqref{eq:ci2} we obtain the 
commutative diagram, with $i\ge 0$ 
\begin{equation}\label{d:1}
\begin{xymatrix}{
& & & 0\\
& & &\coker(\mu)\ar@{^{(}->}[r]\ar[u]& H^1(\cL^{\otimes i})\\
0\ar[r]& I_X(i+2)\ar[r]& \Sym^{i+2}(H^0(\cL)) \ar[r]^{\varphi} & H^0(\cL^{\otimes i+2})\ar[r]\ar[u]^p& 0\\ 
0\ar[r]&\gen{\eta_1,\eta_2}\otimes I_X(i+1) \ar[r] \ar[u]& \gen{\eta_1,\eta_2}\otimes \Sym^{i+1}(H^0(\cL)) 
\ar[r] \ar[u]& \gen{\eta_1,\eta_2}\otimes H^0(\cL^{\otimes i+1}) \ar[r] \ar[u]^\mu& 0.\\
& & & H^0(\cL^{\otimes i})\ar[u]\\
& & & 0\ar[u]}
\end{xymatrix}
\end{equation}
Chasing in it, if we take an 
$\omega\in \Sym^{i+2}(H^0(\cL))$, 
and denoting by $[\omega]$ its 
class in $A_{i+2}$, we have that $[\omega]\neq 0$ if and only if 
$p\circ\varphi(\omega)\in\coker(\mu)\setminus\{0\}$.  
In other words,
\begin{equation}
\coker(\mu)\cong A_{i+2}
\end{equation}
in Diagram~\eqref{d:1}. 
It follows that, if $i>2$, $A_{i+2}=\{0\}$, since $H^1(\cL^{\otimes i})=\{0\}$. 

If $i=2$, we have that $\coker(\mu)\cong H^1(\omega_C)$, by \eqref{eq:c2}, 
so $\dim\coker(\mu)=1$, and therefore $\dim(A_4)=1$. 

If $i=1$, $A_3\cong \CC^{N-1}$ either by the Artinian property, or by Diagram~\eqref{d:1} together with 
Sequence~\eqref{eq:c1}. 

If $i=0$, $A_2\cong \CC^{g-2N-1}$, again by Diagram~\eqref{d:1} and Sequence~\eqref{eq:c0}.


\end{proof}

Now we give the analogue statement of the above one for the general 
$s$-subcanonical curve, 
but, since we do not really need to compute 
the complete Hilbert function we can give a simpler proof suggested 
to us by G. Casnati. 


\begin{thm}\label{thm:4s}
Let $(C,\cL)$ be an $s$-subcanonical curve, $s\ge 2$, embedded by the very ample, complete, 
linear system $\abs{\cL}$ in $\p^N$. 
Let $\eta_1$ and $\eta_2$ two general linear forms on $\p^N$. Let us suppose that $C$ is PN. 
Then, $C$ is aG and the Artinian Gorenstein graded $\CC$-algebra $A:=\frac{S_C}{\gen{\eta_1,\eta_2}}$
is of socle dimension one and degree $s+2$. 
\end{thm}

\begin{proof}
In fact, by Proposition~\ref{prop:pn}, $S_C$ is aG
(here and in what follows, by abuse of notation, 
we will call aG the homogeneous coordinate ring of an aG scheme).

Therefore, also $A=S_X/(\eta_1,\eta_2)$ is aG since 
$\eta_1,\eta_2$ is a regular sequence: see for  example 
\cite[Proposition 3.1.19(b)]{BH}.
The ring $A$ is obviously graded, then by Proposition 
\ref{senzanome}\eqref{senzanome3} it has symmetric Hilbert 
function since it is aG. By symmetry, the socle of $A$ is of 
dimension $1$. Now, 
it remains to prove that the socle of $A$ is of degree  $s+2$. Let $K_{A}$ be the \emph{canonical model} of $A$, see 
\cite[Definition 3.6.8 page 139, also page 140]{BH} and let 
$a(A)$ be the \emph{$a$-invariant} of $A$, 
see \cite[Definition 3.6.13]{BH}. 
By Proposition \ref{senzanome}\eqref{senzanome2}, $K_C=S_C(s)$, then
$K_A=A(s+2)$ since \cite[Corollary 3.6.14]{BH}. In particular, 
$a(A)=s+2$ by \cite[Corollary 3.6.14]{BH}.
This means $A_{s+2}\neq 0$ and 
$A_i=0$ for  $i\ge s+3$ 
(see the remark which follows \cite[Theorem 3.6.19]{BH}. 
\end{proof}

\begin{rmk}
With simple but tedious calculations we could find the values of the Hilbert function of $A$ of the 
preceding theorem. 
\end{rmk}

\subsection{Plane curves}
There are various ways to obtain half-canonical (or---more generally---subcanonical) curves. 
One of these is via (smooth) plane curves and Veronese embeddings of the plane. 

So, let us consider a (non-hyperelliptic) 
plane curve $C$ of odd degree $2n+1$; 
therefore, $C\in\abs{(2n+1)H}$, where $H$ is the hyperplane divisor on $\p^2$. 
By adjunction, we have, by abuse of notation 
\begin{equation}\label{eq:2n}
\omega_C=2(n-1)H{\mid_C}.
\end{equation}
By the Clebsch formula, 
its genus is 
$g(C)=n(2n-1)$. 

Now, consider the $(n-1)$-th embedding of $\p^2$, 
$v_{n-1}\colon\p^2\to\p^N,$ 
where $N:=\binom{n+1}{2}-1$. $V_{n-1}:=v_{n-1}(\p^2)$ is a \emph{Veronese surface}, it is 
PN of degree $(n-1)^2$. 

By the projective normality of $V_{n-1}$ it follows that
$X:=v_{n-1}(C)\subset V_{n-1}$ is PN also: 

\begin{prop} 
Let $C$ be a smooth plane curve of odd degree $2n+1$, and let $X$ be its $(n-1)$-tuple Veronese embedding 
$X:=v_{n-1}(C)\subset V_{n-1}\subset \p^N$, with $N=\binom{n+1}{2}-1$. Then, $X$ is PN and aG. 
\end{prop}

\begin{proof} 
We need to show that 
$h^1(\II_X(j))=0$ for all $j\in \mathbb N$. 

By the above inclusions we can
construct the following exact diagram of sheaves:
\begin{equation*}
\begin{CD}
@.0@.@. @.\\
@.@AAA@. @.\\
@.\II_{X,V_{n-1}}(j)@.@. 0@. \\
@.@AAA@. @AAA\\
0 @>>> \II_X(j) @>>> \OO_{\p^N}(j) @> >> \OO_{X}(j) @>>> 0\\ 
@. @AAA @| @AAA\\
0 @>>> \II_{V_{n-1}}(j) @>>> \OO_{\p^N}(j) 
@>>> \OO_{V_{n-1}}(j) @>>> 0\\ 
@.@AAA@. @AAA\\
@.0@.@. \OO_{V_{n-1}}(j-X) \\
@.@.@. @AAA\\
@.@.@. 0
\end{CD}
\end{equation*}
By the above diagram, 
$\II_{X,V_{n-1}}(j)\cong \OO_{V_{n-1}}(j-X)$ and, since $V_{n-1}$ is PN, then $h^1(\II_{V_{n-1}}(j))=0$ 
for all $j\in \mathbb N$; 
therefore it is sufficient to show that $h^1(\OO_{V_{n-1}}(j-X))=0$ for all $j\in \mathbb N$. 

Since 
$v_{n-1}\colon \p^2
\to V_{n-1}\subset \p^N$ 
is an embedding, we are
reduced to show only that $h^1(\OO_{\p^2}(j(n-1)-2n-1))=0$ for all $j\in \mathbb N$, which is well-known (see 
\cite[III.5.1]{H}).

Moreover, by 
\eqref{eq:2n}, $X$ is also aG. 
\end{proof}

By what we have just proven, we have that 

\begin{prop}\label{cuginettoo} Let $C$ be a plane curve of degree $2n+1$. Let $A$ be the
corresponding Artinian graded Gorenstein ring of the half-canonical curve $X:=v_{n-1}(C)$. 
Then the Macaulay polynomial of $A$ is a quartic of Waring number at most $(n-1)^2$. 
\end{prop}

It is now obvious how to obtain similar results for $s$-subcanonical curves coming from plane curves: 

\begin{thm}\label{cuginetto} Let $C$ be a plane curve of degree $s(n-1)+3$. Let $A$ be the
corresponding Artinian graded Gorenstein ring of $s$-subcanonical curve $X:=v_{n-1}(C)$. 
Then the Macaulay polynomial of $A$ is an $(s+2)$-tic of Waring number at most $(n-1)^2$. 
\end{thm}


\subsection{Curves on Segre-Hirzebruch surfaces}

For fixing some notation, we recall some basic facts about rational normal scrolls.

\subsubsection{Rational normal scrolls}
By definition, a \emph{rational normal scroll} (RNS for short, in the following) 
of type $(a_1,\dotsc,a_k)$, $S_{a_1,\dotsc,a_k}$, is the 
image of the $\p^{k-1}$-bundle $\p(E)=\p(\OO_{\p^1}(a_1)\oplus\dotsb\oplus\OO_{\p^1}(a_k))$, via 
the embedding given by $\OO_{\p(E)}(1)$ in $\p^N$, $N=\sum a_i+k-1$. Equivalently, one takes $k$ disjoint 
projective spaces of dimension $a_i$, $\p^{a_i}$, and $k$ rational normal curves $C_i\subset\p^{a_i}$, 
together with isomorphisms $\phi_i\colon\p^1\to C_i$ (if $a_i\neq 0$, constant maps otherwise); 
then
\begin{equation*}
S_{a_1,\dotsc,a_k}=\bigcup_{P\in\p^1}\gen{\phi_1(P),\dotsc,\phi_k(P)}.
\end{equation*}
We have also that
\begin{align*}
\deg (S_{a_1,\dotsc,a_k}) &=\sum a_i\\
&=N-k+1,
\end{align*}
and $\dim (S_{a_1,\dotsc,a_k})=k$. 

\subsubsection{Rational normal surfaces}

In the case $k=2$, $\p^{a_1+a_2+1}\supset S_{a_1,a_2}\cong \FF_e$, $e:=\abs{a_1-a_2}$ where
$\FF_e$ is its minimal model, 
\ie a Segre-Hirzebruch surface: $\FF_e=\p(\OO_{\p^1}\oplus
\OO_{\p^1}(e))$. In fact, letting $a_{1}\geq a_{2}$, $S_{a_1,a_2}$ is the embedding of $\FF_e$ by the complete linear system 
$\abs{C_0+a_1f}$ where the \emph{zero section} 
$C_0$ on $\FF_e$ is defined by 
$\OO_{\FF_e}(C_{0})=\OO_{\FF_e}(1)$ and $f$ is the fibre of the 
scroll (see for example \cite[V.2]{H}). Moreover, we have 
$C_0\cdot C_0 = -e$, $f\cdot f =0$, $C_0\cdot f =1$, 
and the Picard group of  $\FF_e$ is generated by $C_0$ and $f$. Geometrically
 $C_0$ corresponds in $S_{a_{1}a_{2}}\subset \p^{a_1+a_2+1}$ to the unisecant 
$a_2$-tic $C_{a_2}$, 
while the section $C_1\sim C_0+cf$ (where $\sim$ denotes the linear equivalence of 
divisors) corresponds to the unisecant $a_1$-tic curve. 
We recall moreover that: 
\begin{equation*}
K_{S_{a_1,a_2}}\sim -2C_0+(-2-e)f.
\end{equation*}
Finally, we prove the following

\begin{lem}\label{lem:pic}
Let $C\subset S_{a_1,a_2}\subset\p^{a_1+a_2+1}$ be a curve which is a very ample divisor, \ie $[C]\sim aC_0+bf$ with 
$a>0$ and $b>ae$ (see \cite[Corollary V.2.18]{H}). Then the natural restriction map $\pic(S_{a_1+a_2})\to \pic(C)$ 
restricted to the very ample divisors on $S_{a_1,a_2}$ is injective. 
\end{lem}

\begin{proof}
Let $D_1,D_2\in\pic(S_{a_1,a_2})$ be two very ample divisors; then, we can write
\begin{equation}\label{va}
D_i=\alpha_iC_0+\beta_i f,\quad \alpha_1>0,\ \beta_i>\alpha_ie,\quad i=1,2. 
\end{equation}
Let us suppose that $D_1\mid_C\sim D_2\mid_C$; in particular 
$D_1C=D_2C$, which means 
$a((\beta_1-\alpha_1e)-(\beta_2-\alpha_2e))+b(\alpha_1-\alpha_2)=0$. 
This is equivalent to
\begin{align*}
\alpha_1-\alpha_2&=-\lambda a\\
(\beta_1-\alpha_1e)-(\beta_2-\alpha_2e)&=\lambda b
\end{align*}
with $\lambda\in \CC$. If $\lambda=0$, we are done. If $\lambda> 0$, then $\alpha_1<\alpha_2$ and $\beta_1-\alpha_1e>\beta_2-\alpha_2e$, 
but this contradicts \eqref{va}; analogously if $\lambda< 0$. 
\end{proof}

\subsubsection{Half-canonical curves}
Let us consider now curves contained in rational normal surfaces; we
look for conditions 
for these curves to be half-canonical. We will show that these are characterised to be 
\emph{$4$-gonal}. We start with the following:

\begin{prop}\label{prop:4gona}
If a curve $C\subset S_{a_1,a_2}\subset\p^{a_1+a_2+1}$ (with $a_1\ge a_2$) is half-canonical, then it is $4$-gonal. 
More precisely, $C$ is linearly equivalent to $4C_0+(3a_1-a_2+2) f$, and therefore $p_a(C)=3(a_1+a_2+1)$ and 
$\deg(C)=3a_1+3a_2+2$. $C$ is smooth if (and only if) $3a_2+2\ge a_1$.
\end{prop}

\begin{proof}
A curve $C\subset S_{a_1,a_2}$ determines a divisor $C\sim aC_0+bf$. 
By adjunction, we have $K_C=((a-2)C_0+(b-2-e)f)\mid_C$, 
so $C$ is half-canonical with respect to the linear
system $\abs{C_{0}+a_{1}f}$ if it holds that, by Lemma~\ref{lem:pic}: 
\begin{equation*}
2C_0+2a_1f=((a-2)C_0+(b-2-e)f).
\end{equation*}
This is equivalent to write: $a=4$, $b=3a_1-a_2+2$.
Since we want that $C$ is smooth, by \cite[V.2.18]{H} we must have 
$3a_1-a_2+2\ge 4(a_1-a_2)$, that is $3a_2+2\ge a_1$. Then 
the (arithmetic) genus $p_a(C)$ of $C$ is $p_a(C)=3(a_1+a_2+1)$, 
and its degree is $\deg(C)=3a_1+3a_2+2$. 

\end{proof}

Next, we show that the smooth curves of the preceding proposition are PN: 

\begin{prop}\label{miazia} 
Let $a_{1},a_{2}\in\mathbb N$ such that
$a_{2} \leq a_{1}\leq 3a_{2}+2$. Let $C\subset \FF_e$ be a
general member of the linear system $\abs{4C_{0}+(3a_1-a_2+2)f}$ where
$e=a_{1}-a_{2}\geq 0$. 
Then the image $X$ of $C$ by the embedding $\phi_{\abs{C_{0}+a_{1}f}}\colon \FF_e
\rightarrow S_{ a_{1},a_{2}}\subset\p^{1+a_{1}+a_{2}}$ is PN.
\end{prop}

\begin{proof} Set $N:=1+a_{1}+a_{2}$, $S=S_{ a_{1},a_{2}}\subset
\mathbb P^N$ and $X= \phi_{\abs{C_{0}+a_{1}f}}(C)$. We only need to show that 
$h^1(\II_X(j))=0$ for all $j\in \mathbb N$. 

By the natural inclusions $X\subset S\subset\p^N$ we can
construct the following exact diagram of sheaves:

\begin{equation}\label{d:3}
\begin{CD}
@.0@.@. @.\\
@.@AAA@. @.\\
@.\II_{X,S}(j)@.@. 0@. \\
@.@AAA@. @AAA\\
0 @>>> \II_X(j) @>>> \OO_{\p^N}(j) @> >> \OO_{X}(j) @>>> 0\\ 
@. @AAA @| @AAA\\
0 @>>> \II_{S}(j) @>>> \OO_{\p^N}(j) 
@>>> \OO_{S}(j) @>>> 0\\ 
@.@AAA@. @AAA\\
@.0@.@. \OO_{S}(j-X) \\
@.@.@. @AAA\\
@.@.@. 0
\end{CD}
\end{equation}
By the above diagram, 
$\II_{X,S}(j)\cong \OO_{S}(j-X)$ and, since $S$ is PN, then $h^1(\II_S(j))=0$ for all $j\in \mathbb N$; 
therefore it is sufficient to show that $h^1(\OO_S(j-X))=0$ for all $j\in \mathbb N$. 

Since 
$\phi_{\abs{C_{0}+a_{1}f}}\colon \FF_e
\to S\subset \p^N$ 
is an embedding, we are
reduced to show only that $h^1(\OO_{\FF_{e}}(j(C_{0}+a_{1}f)-C))=0$ for all $j\in \mathbb N$.

Since
\begin{equation*}
C\in \abs{4C_{0}+(3a_1-a_2+2)f}\quad \textup{and}\quad K_{\FF_{e}}\sim
-2C_0+(-2-e)f, 
\end{equation*}
then 
\begin{equation*}
j(C_{0}+a_{1}f)-C\sim 
K_{\FF_{e}}+(j-2)(C_{0}+a_{1}f). 
\end{equation*}

If $j\geq 3$, then
\begin{align*}
H^1(\OO_{\FF_{e}}(j(C_{0}+a_{1}f)-C))&=H^1(\OO_{\FF_{e}}(K_{\FF_{e}}+(j-2)(C_{0}+a_{1}f)))\\
&=\{0\}
\end{align*}
by the Kodaira Vanishing theorem \cite[III.7.15]{H}. 

If $j=2$, then $h^1(\OO_{\FF_{e}}(K_{\FF_{e}}))=0$ 
since $\FF_{e}$ is a rational surface.

If $j=1$, 
by Serre duality 
\begin{equation*}
H^1(\OO_{\FF_{e}}(K_{\FF_{e}}-(C_{0}+a_{1}f)))\cong H^1(\OO_{\FF_{e}}(C_{0}+a_{1}f)))^{\vee},
\end{equation*}
so we only need to show that $h^1(\OO_{\FF_{e}}(C_{0}+a_{1}f))=0$.
Now, the general member $H\in \abs{C_{0}+a_{1}f}$ is a connected smooth
rational curve and its degree is $a_{1}+a_{2}$, since $H$ corresponds to a hyperplane section of $S$. 
Then $h^{1}(H, \OO_{H}(C_{0}+a_{1}f))=0$ by Serre duality on the curve $H$. 

By the cohomology of 
\begin{equation*}
0\to \OO_{\FF_{e}} \to \OO_{\FF_{e}}(C_{0}+a_{1}f) \to
\OO_{H}(C_{0}+a_{1}f) \to 0, 
\end{equation*}
since $h^1(\OO_{\FF_{e}})=0$ ($\FF_{e}$ is a rational, hence regular, surface) and $h^{1}(H, \OO_{H}(C_{0}+a_{1}f))=0$, 
it follows that $h^1\OO_{\FF_{e}}(C_{0}+a_{1}f)=0$, which is the claim.

Finally, if $j=0$, we have $h^1\OO_{\FF_{e}}(-C)=0$ by Kodaira Vanishing. 

\end{proof}


In the hypothesis of Proposition \ref{miazia}, we can give the converse of Proposition~\ref{prop:4gona}, and therefore 
we can characterise the PN $4$-gonal half-canonical curves, in the following 

\begin{thm}\label{miofratello} 
Let $(C,\cL)$ be a polarised curve, such that $C\subset\abs{\cL}^\vee=:\check\p^N$ is a half-canonical PN curve. 
Then $C$ is $4$-gonal 
if and only if it is contained in a rational normal surface $S_{a_1,a_2}$ with $a_{2} \leq a_{1}\leq 3a_{2}+2$ and $N:=a_1+a_2+1$, if and only if 
for every
couple of general sections
$\eta_1, \eta_2\in H^{0}(C,\cL)$, $F_{\eta_{1},\eta_{2}}\in \mathbb C[x_{0},\dotsc, x_{a_{1}+a_{2}-1}]$ 
is a Fermat quartic. 
\end{thm}

\begin{proof} 
From Proposition~\ref{prop:4gona} we deduce that if $C\subset S_{a_1,a_2}$ and it is PN, then $C$ is $4$-gonal and $a_{2} \leq a_{1}\leq 3a_{2}+2$. 
If we take two general sections $\eta_1, \eta_2\in H^{0}(C,\cL)$, the zero-dimensional scheme of length $a_1+a_2$, 
$\Gamma:=S_{a_1,a_2}\cap V(\eta_1,\eta_2)\subset\check\p^{a_1+a_2-1}$ is apolar to 
a Fermat quartic hypersurface, $F_{\eta_1,\eta_2}\in \mathbb C[x_{0},\dotsc, x_{a_{1}+a_{2}-1}]$ 
by the Apolarity Lemma~\ref{lem:apo}, since $S_{a_1,a_2}$ is aCM, 
and by Proposition~\ref{prop:4}.

Now, let us suppose that our half-canonical curve $C$ is such that for every
couple of general sections
$\eta_1, \eta_2\in H^{0}(C,\cL)$, there is a zero-dimensional scheme of length $a_1+a_2$, 
$\Gamma_{\eta_1,\eta_2}\subset \p^{a_1+a_2-1}:=V(\eta_1, \eta_2)$ with $I(\Gamma_{\eta_1,\eta_2})\subset I(C)$, or, 
in other words, by the Apolarity Lemma~\ref{lem:apo}, $\Gamma_{\eta_1,\eta_2}$ is apolar to a Fermat quartic 
$F_{\eta_1,\eta_2}\in \mathbb C[x_0,\dotsc,
x_{a_1+a_2-1}]$. 

First of all, we can assume that $\eta_1=\de_{N-1}$ and $\eta_2=\de_N$ and 
$F_{\de_{N-1},\de_N}:=x_0^4+\dotsb+x_{N-2}^4$. 
We only need to find 
$F_{\de_{N-1},\de_N}^\perp$. 
It is easy to see that 
\begin{equation}\label{eq:riducib}
F_{\de_{N-1},\de_N}^\perp=(\de_i\de_j, \de_i^4-\de_j^4),
\quad i,j\in\{0,\dotsc,N-2\},\quad i\neq j. 
\end{equation}
Then, the quadrics of $\II(C)$ are of the form 
\begin{equation}\label{eq:quadrij}
Q_{i,j}:=\de_i\de_j+\de_{N-1}L_{i,j}+\de_{N}M_{i,j}, 
\end{equation}
where $L_{i,j}$ and $M_{i,j}$ are linear forms on $\check\p^N$. 
Therefore, we have a set of $\binom{N-2}{2}$ quadrics in $\II(C)$, and $\II(C)$ is 
not generated by quadrics.

Now, if the ideal 
\begin{equation*}
I:=(Q_{i,j})_{0\le i < j\le N-2}
\end{equation*}
defines a surface $S=V(I)$, then we are done, since $\Gamma:=S\cap \check\p^{N-2}$---where 
$\check\p^{N-2}:=V(\de_{N-1},\de_N)$---is a zero-dimensional scheme of length $N-1$, in fact
\begin{equation*}
I(\Gamma)=(\de_i\de_j)_{0\le i < j\le N-2},
\end{equation*}
and therefore $S$ is a surface of minimal degree, hence a rational normal surface.

Now, we have that 
\begin{equation*}
I(C)=(Q_{i,j}, \de_i^4-\de_j^4+\de_{N-1} N_{i,j}+\de_NR_{i,j})_{0\le i < j\le N-2}
\end{equation*}
with $N_{i,j}$ and $R_{i,j}$ homogeneous cubic forms of $\check\p^N$, and by hypothesis, if $\eta_1$ and $\eta_2$ are 
two general linear forms, Equation~\eqref{eq:riducib} becomes 
\begin{equation*}
F_{\eta_{1},\eta_{2}}^\perp=(\ell_i\ell_j, \ell_i^4-\ell_j^4),
\quad i,j\in\{0,\dotsc,N-2\},\quad i\neq j. 
\end{equation*}
with $\ell_i$ linear forms, and 
\begin{equation*}
Q_{i,j}':=\ell_i\ell_j+\eta_1L_{i,j}'+\eta_2M_{i,j}', 
\end{equation*}
are the quadrics in $I(C)$; but the only way to obtain these quadrics is from linear combinations of the $Q_{i,j}$'s, so 
\begin{equation*}
I=(Q_{i,j}')_{0\le i < j\le N-2}. 
\end{equation*}
Now, $\Gamma':=S\cap \check\p^{N-2'}$---where 
$\check\p'^{N-2}:=V(\eta_1,\eta_2)$---is a zero-dimensional scheme of length $N-1$, in fact, again 
\begin{equation*}
I(\Gamma')=(\ell_i\ell_j)_{0\le i < j\le N-2},
\end{equation*}
and therefore $S$ is a surface of minimal degree, hence a rational normal surface, and the proof is complete.


\end{proof}

\begin{rmk}\label{rmk:inutil}
It is not difficult to find that the quadrics $Q_{i,j}$ of Equation~\eqref{eq:quadrij} can be written in a particular form. 
In fact, if we choose two linear forms one of 
which is---for example---$\eta_1=\de_i$, 
we deduce, in $F_{\eta_{1},\eta_{2}}^\perp$, 
\begin{equation}\label{eq:ride}
\de_{N-1}L_{i,j}+\de_{N}M_{i,j}=\ell_i m_j, 
\end{equation}
where $\ell_i$ and $m_j$ are linear forms in the appropriate $\check\p^{N-2}:=V(\{\eta_1=\eta_2=0\})$. 
By the generality of the two linear forms $\eta_{1},\eta_{2}$, it follows, by Bertini Theorem, if $N>3$, that the quadric 
$\de_{N-1}L_{i,j}+\de_{N}M_{i,j}$ is reducible, and Equation~\eqref{eq:quadrij} can be written as
\begin{equation*}
Q_{i,j}=\det
\begin{pmatrix}\de_i &L_i\\
M_j &\de_j
\end{pmatrix}, 
\end{equation*}
where $L_i$ and $M_j$ are linear forms in $\p^N$ that determine $\ell_i$ and $m_j$ in Equation~\eqref{eq:ride}. Actually, it is not hard to prove that, always 
by Equation~\eqref{eq:quadrij}, that $L_i$ (or $M_j$) is a linear form 
of the type
\begin{equation*}
L_j=a\de_{N-1}+b\de_N.
\end{equation*}
\end{rmk}






\subsubsection{$s$-subcanonical curves}
Let us make the natural extension to the $s$-subcanonical curves contained in rational normal surfaces; 
from what we know from the last subsection, not surprisingly, 
we will see that these are characterised to be 
\emph{$(s+2)$-gonal}. We start with the following

\begin{prop}\label{prop:4gonas}
If a curve $C\subset S_{a_1,a_2}$ (with $a_1\ge a_2$) is $s$-subcanonical, then it is $s+2$-gonal; 
more precisely, $C$ is linearly equivalent to $(s+2)C_0+((s+1)a_1-a_2+2) f$, and therefore $p_a(C)=\frac{s+1}{2}(s(a_1+a_2)+2)$ and 
$\deg(C)=(s+1)(a_1+a_2)+2$. $C$ is smooth if (and only if) $(s+1)a_2+2\ge a_1$.
\end{prop}

\begin{proof}
We proceed as in the proof of Proposition~\ref{prop:4gona}: 
$C$ determines a divisor 
$C\sim aC_0+bf$ and it is is $s$-subcanonical with respect to the linear
system $\abs{C_{0}+a_{1}f}$ if it holds that, by adjunction and Lemma~\ref{lem:pic} 
\begin{equation*}
sC_0+sa_1f=((a-2)C_0+(b-2-e)f);
\end{equation*}
that is
\begin{align*}
a&=s+2\\
b&=(s+1)a_1-a_2+2. 
\end{align*}
Since we want $C$ to be smooth, then $(s+1)a_2+2\ge a_1$. 
The (arithmetic) genus $p_a(C)$ of $C$ is given, by adjunction 
$p_a(C)=\frac{(s+1)(s(a_1+a_2)+2)}{2}$, and its degree is 
$\deg(C)=(s+1)(a_1+a_2)+2$. 
\end{proof}

Next, we show that the smooth curves of the preceding proposition are PN: 

\begin{prop}\label{miazias} 
Let $a_{1},a_{2}\in\mathbb N$ such that
$(s+1)a_2+2\ge a_1\ge a_2$. Let $C\in \FF_e$ be a
general member of the linear system $\abs{(s+2)C_0+((s+1)a_1-a_2+2) f}$ where
$e=a_{1}-a_{2}\geq 0$. 
Then the image $X$ of $C$ by the embedding $\phi_{\abs{C_{0}+a_{1}f}}\colon \FF_e
\rightarrow S_{ a_{1},a_{2}}\subset\p^{1+a_{1}+a_{2}}$ is PN.
\end{prop}

\begin{proof} We proceed as in the proof of Proposition~\ref{miazia}, with the same notations. 
From Diagram~\eqref{d:3} 
it is sufficient to show that $h^1(\OO_S(j-X))=0$ for all $j\in \mathbb N$, and then to show that 
$h^1(\OO_{\FF_{e}}(j(C_{0}+a_{1}f)-C))=0$ for all $j\in \mathbb N$. 

Now, we have 
\begin{equation*}
j(C_{0}+a_{1}f)-C\sim 
K_{\FF_{e}}+(j-s)(C_{0}+a_{1}f). 
\end{equation*}

If $j\geq s+1$, then
\begin{align*}
H^1(\OO_{\FF_{e}}(j(C_{0}+a_{1}f)-C))&=H^1(\OO_{\FF_{e}}(K_{\FF_{e}}+(j-s)(C_{0}+a_{1}f)))\\
&=\{0\}
\end{align*}
by the Kodaira Vanishing theorem \cite[III.7.15]{H}. 

If $j=s$, then 
$h^1(\OO_{\FF_{e}}(K_{\FF_{e}}))=0$ 
since $\FF_{e}$ is a rational surface. 


If $j=s-1$, 
by Serre duality 
\begin{equation*}
H^1(\OO_{\FF_{e}}(K_{\FF_{e}}-(C_{0}+a_{1}f)))\cong H^1(\OO_{\FF_{e}}(C_{0}+a_{1}f)))^{\vee},
\end{equation*}
so we only need to show that $h^1(\OO_{\FF_{e}}(C_{0}+a_{1}f))=0$.
Now, the general member $H\in \abs{C_{0}+a_{1}f}$ is a connected smooth
rational curve and its degree is $a_{1}+a_{2}$, since $H$ corresponds to a hyperplane section of $S$. 
Then $h^{1}(H, \OO_{H}(C_{0}+a_{1}f))=0$ by Serre duality on the curve $H$. 

By the cohomology of 
\begin{equation*}
0\to \OO_{\FF_{e}} \to \OO_{\FF_{e}}(C_{0}+a_{1}f) \to
\OO_{H}(C_{0}+a_{1}f) \to 0, 
\end{equation*}
since $h^1(\OO_{\FF_{e}})=0$ ($\FF_{e}$ is a rational, hence regular, surface) and $h^{1}(H, \OO_{H}(C_{0}+a_{1}f))=0$, 
it follows that $h^1\OO_{\FF_{e}}(C_{0}+a_{1}f)=0$, which is the claim.

If $j\le s-2$, we set $j=s-2-i$, with $0\le i\le s-2$. 
Again by Serre duality we have to show that 
$h^1(\OO_{\FF_{e}}((i+2)(C_{0}+a_{1}f)))=0$. 

As above the general member $H\in \abs{(i+2)(C_{0}+a_{1}f)}$ is a connected smooth
curve of degree $(i+2)(a_{1}+a_{2})$. Its genus $g$ is given by adjunction:
\begin{align*}
2g-2&=(i+2)(C_{0}+a_{1}f)((i+2)(C_{0}+a_{1}f)+K_{\FF_{e}})\\
&=(i+2)(i+1)(a_1+a_2)-2(i+2),
\end{align*}
which means
\begin{equation*}
g=\binom{i+2}{2}\deg(S)-(i+1).
\end{equation*}

Then, since $(i+2)^2(\deg(S))^2> 2g-2$, we deduce $h^{1}(H, \OO_{H}((i+2)(C_{0}+a_{1}f))=0$.

Finally, considering the cohomology of 
\begin{equation*}
0\to \OO_{\FF_{e}}((i+1)C_{0}+a_{1}f) \to \OO_{\FF_{e}}((i+2)(C_{0}+a_{1}f)) \to
\OO_{H}((i+2)(C_{0}+a_{1}f)) \to 0, 
\end{equation*}
since we have just proved that $h^1(\OO_{\FF_{e}}(C_{0}+a_{1}f))=0$ and $h^{1}(H, \OO_{H}((i+2)(C_{0}+a_{1}f))=0$, 
$\forall i$, $0\le i\le s-2$, we deduce, by induction on $i$, that 
$h^1(\OO_{\FF_{e}}((i+2)(C_{0}+a_{1}f)))=0$, $\forall i$, $0\le i\le s-2$, and therefore the proposition is proved. 
\end{proof}


In the hypothesis of Proposition \ref{miazias}, we can give the converse of Proposition~\ref{prop:4gonas}, and therefore 
we can characterise the PN $(s+2)$-gonal half-canonical curves, in the following 

\begin{thm}\label{miofratellos} 
Let $(C,\cL)$ be a polarised curve, such that $C\subset\abs{\cL}^\vee=:\check\p^N$ is a $s$-subcanonical PN curve. 
Then $C$ is $(s+2)$-gonal 
if and only if it is contained in a rational normal surface $S_{a_1,a_2}$ with $a_{2} \leq a_{1}\leq (s+1)a_{2}+2$ and $N:=a_1+a_2+1$, if and only if 
for every
couple of general sections
$\eta_1, \eta_2\in H^{0}(C,\cL)$, $F_{\eta_{1},\eta_{2}}\in \mathbb C[x_{0},\dotsc, x_{a_{1}+a_{2}-1}]$ 
is a Fermat $(s+2)$-tic. 
\end{thm}

\begin{proof} 
From Proposition~\ref{prop:4gonas} we deduce that if $C\subset S_{a_1,a_2}$ and it is PN, then $C$ is $(s+2)$-gonal and $a_{2} \leq a_{1}\leq (s+1)a_{2}+2$. 
If we take two general sections $\eta_1, \eta_2\in H^{0}(C,\cL)$, the zero-dimensional scheme of length $a_1+a_2$, 
$\Gamma:=S_{a_1,a_2}\cap V(\eta_1,\eta_2)\subset\check\p^{a_1+a_2-1}$ is apolar to 
a Fermat $(s+2)$-tic hypersurface, $F_{\eta_1,\eta_2}\in \mathbb C[x_{0},\dotsc, x_{a_{1}+a_{2}-1}]$ 
by the Apolarity Lemma~\ref{lem:apo}, since $S_{a_1,a_2}$ is aCM, 
and by Theorem~\ref{thm:4s}.

Now, let us suppose that our $s$-subcanonical curve $C$ is such that for every
couple of general sections
$\eta_1, \eta_2\in H^{0}(C,\cL)$, there is a zero-dimensional scheme of length $a_1+a_2$, 
$\Gamma_{\eta_1,\eta_2}\subset \p^{a_1+a_2-1}:=V(\eta_1, \eta_2)$ with $I(\Gamma_{\eta_1,\eta_2})\subset I(C)$, or, 
in other words, by the Apolarity Lemma~\ref{lem:apo}, $\Gamma_{\eta_1,\eta_2}$ is apolar to a Fermat $(s+2)$-tic 
$F_{\eta_1,\eta_2}\in \mathbb C[x_0,\dotsc,
x_{a_1+a_2-1}]$. 

First of all, we can assume that $\eta_1=\de_{N-1}$ and $\eta_2=\de_N$ and 
$F_{\de_{N-1},\de_N}:=x_0^4+\dotsb+x_{N-2}^4$. 
We only need to find 
$F_{\de_{N-1},\de_N}^\perp$. 
It is easy to see that 
\begin{equation}\label{eq:riducibs}
F_{\de_{N-1},\de_N}^\perp=(\de_i\de_j, \de_i^{s+2}-\de_j^{s+2}),
\quad i,j\in\{0,\dotsc,N-2\},\quad i\neq j. 
\end{equation}
Then, the quadrics of $\II(C)$ are of the form 
\begin{equation}\label{eq:quadrijs}
Q_{i,j}:=\de_i\de_j+\de_{N-1}L_{i,j}+\de_{N}M_{i,j}, 
\end{equation}
where $L_{i,j}$ and $M_{i,j}$ are linear forms on $\check\p^N$. 
Therefore, we have a set of $\binom{N-2}{2}$ quadrics in $\II(C)$, and $\II(C)$ is 
not generated by quadrics.

Now, if the ideal 
\begin{equation*}
I:=(Q_{i,j})_{0\le i < j\le N-2}
\end{equation*}
defines a surface $S=V(I)$, then we are done, since $\Gamma:=S\cap \check\p^{N-2}$---where 
$\check\p^{N-2}:=V(\de_{N-1},\de_N)$---is a zero-dimensional scheme of length $N-1$, in fact
\begin{equation*}
I(\Gamma)=(\de_i\de_j)_{0\le i < j\le N-2},
\end{equation*}
and therefore $S$ is a surface of minimal degree, hence a rational normal surface.

Now, we have that 
\begin{equation*}
I(C)=(Q_{i,j}, \de_i^{s+2}-\de_j^{s+2}+\de_{N-1} N_{i,j}+\de_NR_{i,j})_{0\le i < j\le N-2}
\end{equation*}
with $N_{i,j}$ and $R_{i,j}$ homogeneous $s-1$-forms of $\check\p^N$, and by hypothesis, if $\eta_1$ and $\eta_2$ are 
two general linear forms, Equation~\eqref{eq:riducibs} becomes 
\begin{equation*}
F_{\eta_{1},\eta_{2}}^\perp=(\ell_i\ell_j, \ell_i^{s+2}-\ell_j^{s+2}),
\quad i,j\in\{0,\dotsc,N-2\},\quad i\neq j. 
\end{equation*}
with $\ell_i$ linear forms, and 
\begin{equation*}
Q_{i,j}':=\ell_i\ell_j+\eta_1L_{i,j}'+\eta_2M_{i,j}', 
\end{equation*}
are the quadrics in $I(C)$; but the only way to obtain these quadrics is from linear combinations of the $Q_{i,j}$'s, so 
\begin{equation*}
I=(Q_{i,j}')_{0\le i < j\le N-2}. 
\end{equation*}
Now, $\Gamma':=S\cap \check\p^{N-2'}$---where 
$\check\p'^{N-2}:=V(\eta_1,\eta_2)$---is a zero-dimensional scheme of length $N-1$, in fact, again 
\begin{equation*}
I(\Gamma')=(\ell_i\ell_j)_{0\le i < j\le N-2},
\end{equation*}
and therefore $S$ is a surface of minimal degree, hence a rational normal surface, and the proof is complete.


\end{proof}

\begin{rmk}
Clearly the observations done in Remark~\ref{rmk:inutil} on the quadrics of the 
ideal of the half-canonical curves hold in the
general case of the quadrics of the 
ideal of the $s$-subcanonical curves. 
\end{rmk}

\providecommand{\bysame}{\leavevmode\hbox to3em{\hrulefill}\thinspace}
\providecommand{\MR}{\relax\ifhmode\unskip\space\fi MR }
\providecommand{\MRhref}[2]{%
\href{http://www.ams.org/mathscinet-getitem?mr=#1}{#2}
}
\providecommand{\href}[2]{#2}


\begin{thebibliography}{ACGH85}

\bibitem[ACGH85]{ACGH}
Enrico Arbarello, Maurizio Cornalba, Phillip~A. Griffiths, and Joseph Harris,
\emph{Geometry of algebraic curves. {V}ol. {I}}, Grundlehren der
Mathematischen Wissenschaften [Fundamental Principles of Mathematical
Sciences], vol. 267, Springer-Verlag, New York, 1985.


\bibitem[Bab39]{Ba}
D. W. Babbage, \emph{A note on the quadrics through a canonical curve}, J. London Math.
Soc. \textbf{14} (1939), 310–-315. 



\bibitem[BK05]{bk}
Michel Brion and Shrawan Kumar, 
\emph{Frobenius splitting methods in geometry and representation theory}, 
Progress in Mathematics, vol.~231, Birkh\"auser, Boston, Mass., 2005. 

\bibitem[BH93]{BH}
Winfried Bruns and J\"urgen Herzog, 
\emph{Cohen-Macaulay rings}, Cambridge Studies in Advanced Mathematics, 39. 
Cambridge University Press, Cambridge, 1993.  





\bibitem[DZ]{DZ}
Pietro De Poi and Francesco Zucconi, 
\emph{Gonality, apolarity and hypercubics}, 
Preprint \url{arXiv:0802.0705v3 [math.AG]}, submitted. 

\bibitem[Enr19]{E}
Federigo Enriques, \emph{Sulle curve canoniche di genere $p$  dello spazio a $p -1$ dimensioni}, 
Rend. dell'Acc. delle Scienze di Bologna , (n. s.), \textbf{23} (1919), 80--82. 








\bibitem[Har83]{H}
Robin Hartshorne, \emph{Algebraic geometry}, {Graduate Texts in Mathematics},
vol.~52, Springer-Verlag, New York-Heidelberg-Berlin, 1983, {Corr. 3rd
printing}.

\bibitem[IK99]{ik}
Anthony Iarrobino and Vassil Kanev, \emph{Power sums, {G}orenstein algebras,
and determinant loci}, Lect. Notes Math., vol. 1721, Springer, Berlin, 1999.

\bibitem[IR01]{IR}
Atanas Iliev and Kristian Ranestad, \emph{Canonical curves and varieties of
sums of powers of cubic polynomials}, J. Algebra \textbf{246} (2001) no.~1,
385--393.

\bibitem[Mac94]{Mac}
Francis S. Macaulay, 
\emph{The algebraic theory of modular systems. With a new introduction by Paul Roberts}, 
Reprint of the 1916 orig. 
Cambridge Mathematical Library. Cambridge: Cambridge University Press, (1994). 




\bibitem[Mig98]{Mig}
Juan C. Migliore,
\emph{Introduction to liaison theory and deficiency modules}, 
Progress in Mathematics, vol.~165, Birkh\"auser, Boston, Mass., 1998. 

\bibitem[Pet23]{P}
Karl Petri,
\emph{\"Uber die invariante Darstellung algebraischer Funktionen einer Ver\"anderlichen}
Math. Ann., \textbf{88} (1923),  242--289. 

\bibitem[SD73]{SD}
Bernard Saint-Donat, 
\emph{On Petri's analysis of the linear system of quadrics through a canonical curve}, 
Math. Ann. \textbf{206} (1973), 157--175.


\bibitem[Sok71]{So}
V. V. Sokurov, 
\emph{The Noether-Enriques theorem on canonical curves}, 
Math. USSR, Sb. \textbf{15} (1971), 361--403 (1972).







\end{thebibliography}
\end{document}